\documentclass[reqno, 11pt]{amsart}

\usepackage{amscd,amsfonts,amssymb} 
\usepackage[nocompress]{cite} %hyperlink in pdf
\usepackage{graphicx}
\usepackage{amsmath,mathtools} %subscript on left
\usepackage{bm}

\usepackage[ left=3.7cm, right=3.7cm, top=2.5cm, bottom=2.5cm ]{geometry}
\usepackage{pgfplots}
\usepackage[open]{bookmark} %generate outline. generate links to equations and references
\usepackage[nocompress]{cite} %auto ordering number. \cite{3,1,2} to [1,2,3] 
\usepackage{stmaryrd} %double square bracket for jump of functions
\usepackage{bbm} % lower-case mathbb; use \mathbbm{abcd}
\usepackage{url} %makes url in reference better (allow line break; different font)

\numberwithin{equation}{section}

\newtheorem{lemma}{Lemma}[section]

\newtheorem{proposition}{Proposition}[section]

\theoremstyle{definition}

\theoremstyle{remark}
\newtheorem{remark}{Remark}[section]

\usepackage[colorinlistoftodos,prependcaption,textsize=tiny]{todonotes}

\providecommand{\abs}[1]{\left| #1 \right|}   %{\lvert#1\rvert} 
\providecommand{\prts}[1]{ \left( #1 \right) }
\author{   Tian Jing}
\address{Department of Mathematics, University of Michigan, Ann Arbor, MI 48109, USA.}
\email{tnjing@umich.edu}

\title[Self-Similar Solutions]
{Self-Similar Solutions to a Nonlinear Forward-Backward Parabolic Equation}


\numberwithin{equation}{section}
\numberwithin{figure}{section}
\numberwithin{table}{section}

\begin{document}

\begin{abstract}
In this paper, we study the mixed-type equation $uu_x = u_{yy}$,
which behaves as forward and backward parabolic equations depending on the sign of $u$. 
The equation arises from the study of boundary layers with separation.
We seek  solutions that change their type smoothly to better understand the  equation. We simplify the equation into a second-order ODE using similarity variables, and prove an existence result by analyzing it as a first-order nonlinear ODE system. This provides us a self-similar solution with a sign change.  

\end{abstract}

\date{\today}

\maketitle

%\tableofcontents

\section{Introduction}

In this paper, we study the equation
\begin{equation} \label{main PDE}
	uu_x - u_{yy} = 0.
\end{equation}
Our goal is to find a suitable domain and obtain a solution whose  sign changes.
When $u>0$,  $u<0$, and $u=0$, the equation becomes parabolic, backward-parabolic, and degenerate, respectively. In general, it is hard to formulate the problem with suitable boundary conditions. Thus, we need to focus on some special solutions which gives us insights about the property of this equation.
One way is to construct suitable similarity variables to turn the PDE into an ODE. 
By designating a negative initial value and a positive initial derivative, 
we expect the solution to change its sign.

The equation \eqref{main PDE} was introduced by Oleinik  \cite{Oleinik.summerschool} in the study of boundary layer problems.  
In \cite{Oleinik.summerschool}, the equation is studied in a rectangular domain with boundary conditions given on all four sides, as shown in the left part of Figure \ref{figure: Oleinik formulation}. The positive and negative boundary conditions on the left and right sides behave like the initial value for forward and backward parabolic equations. The weak solution for Oleinik's problem is studied in \cite{Bocharov}.
In  \cite{Kuznetsov}, the left and right boundary values are only assumed to be $L^\infty$ without further requirements on their signs.  An entropy solution is obtained using the vanishing viscosity term $\varepsilon u_{xx}$. 

\begin{figure}[h]
	
		\centering
	\begin{minipage}{0.3\textwidth}
		
	\centering
	\begin{tikzpicture}
		\coordinate (A) at (0,0);
		\coordinate (B) at (0,2);
		\coordinate (C) at (3,2);
		\coordinate (D) at (3,0);
		
		\draw[blue, thick] (A) -- (B);
		\draw[black] (B) -- (C);
		\draw[red, thick] (C) -- (D);
		\draw[black] (D) -- (A);
		
		\node at (0.2,1) {\small $+$}; 
		\node at (1.5,1.8) {\small $0$}; 
		\node at (2.8,1) {\small $-$}; 
		\node at (1.5,0.2) {\small $0$}; 
	\end{tikzpicture}
	
\end{minipage}
	\hspace{0.1\textwidth}
	\begin{minipage}{0.3\textwidth}
			\centering
		\begin{tikzpicture}
		\coordinate (A) at (0,0);
		\coordinate (B) at (0,2);
		\coordinate (C) at (3,2);
		\coordinate (D) at (3,0);
		
		\draw[black,dashed] (A) -- (B);
		\draw[black] (B) -- (C);
		
		\draw[black,dashed] (C) -- (D);
		
		\draw[black] (D) -- (A);
		
		\draw[blue, thick] (0,1) -- (B);
		
		\draw[red, thick] (3,1) -- (D);
		
		\node at (0.2,1.4) {\small $+$}; 
		\node at (1.5,1.8) {\small $0$}; 
		\node at (2.8,0.5) {\small $-$}; 
		\node at (1.5,0.2) {\small $0$}; 
	\end{tikzpicture}
\end{minipage}
	\caption{Different boundary conditions}
	\label{figure: Oleinik formulation}
	
\end{figure}
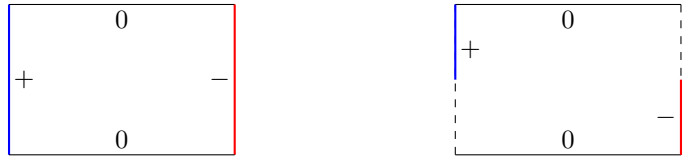

Another formulation   has been studied in \cite{Dalibard.version2025}, where  positive and negative boundary conditions are given partially on the left and right boundaries respectively, as shown in the right part of Figure \ref{figure: Oleinik formulation}. The general boundary layer problem with separation phenomena is studied in \cite{Masmoudi.Reversal,Masmoudi.Higher-regularity}.

If we focus on nonnegative solutions $u\geq 0$, a change of variable $w= u^2$ transforms our problem to the fast diffusion equation $w_x - \prts{w_y / \sqrt{w}}_y =0  $.
The self-similar solution for  $u_x - \prts{u^{s-1} u_y}$ with $s\leq 0$ is studied in \cite{Ferreira.self-similar.very-fast}.
We refer the readers to \cite{Vazquez.book.nonlinear-diffusion, Bonforte.FastDiffusion} for a detailed introduction to fast diffusion equations.

The forward-backward equation with structure $u_x= \prts{\phi(u)}_{yy}$ is studied in \cite{Smarrazzo.forward-backward, Smarrazzo.Sobolev-approximation, Gilding-Tesei.forward-backward.ODE,Mascia-Terracina-Tesei.forward-backward}, etc. 
In \cite{Gilding-Tesei.forward-backward.ODE}, the equation is transformed to an ODE using similarity variables, and all possible types of solutions to this ODE are discussed.
In \cite{Kim-Yan.forward-backward}, equations with principal structure $u_x = \prts{\phi(u_y)}_y $ are studied.

The main difficulty of our problem is the strong nonlinearity. After simplification  to an ODE, the equation is non-autonomous and contains two nonlinear terms with different behavior. We overcome this challenge by finding a suitable simplification and using  elementary ideas from the Picard–Lindel\"of theorem.

The paper will be organized as follows. In Section \ref{Section: transform to ODE}, we transform the PDE into an ODE and apply reasonable simplifications. 
In Section 
\ref{section: solve ODE A=1 B=0}, we prove the existence and uniqueness of solutions to this ODE.
In Section \ref{Section: self-similar solution}, we recover the profile of self-similar solution $u(x,y)$ and discuss some good properties of it.

\section{Transform to ODE}
\label{Section: transform to ODE}

In this section, we transform the equation $uu_x - u_{yy} = 0$ to an ordinary differential equation.

Suppose the solution has the form:
\[
	u(x,y) = \lambda(x) f\prts{\frac{y}{\delta(x)}} =: \lambda(x) f(\eta),
\]
with $\eta$ being the similarity variable
\[
	\eta:= \frac{y}{\delta(x)}.
\]
Computing the derivatives of $u$, we have
\[
	\begin{aligned}
		& u_x = \lambda'(x) f(\eta) + \lambda(x) f'(\eta) \cdot \frac{y \delta'(x)}{-\delta^2(x)}
		\\
		& \quad
		=  \lambda'(x) f(\eta) - \lambda(x) f'(\eta) \eta \frac{ \delta'(x)}{\delta(x)},
		\\
		& 
		u_y = \lambda(x) f'(\eta) \frac{1}{\delta(x)},
		\\
		& 
		u_{yy} = \lambda(x) f''(\eta) \frac{1}{\delta^2(x)}.
	\end{aligned}
\]
We will omit the variables in functions when there is no ambiguity.
Thus, we have
\[
	\lambda f \prts{ \lambda' f - \lambda f' \eta \frac{ \delta'}{\delta} } = \lambda f'' \frac{1}{\delta^2},
\]
which implies
\[
	f'' + \lambda \delta \delta' \cdot f f' \eta - \lambda' \delta^2 \cdot f^2 = 0.
\]
To obtain an equation with the single variable $\eta$, we need to assume 
\begin{equation} \label{simplify to ODE, constant terms}
	\begin{aligned}
		& \lambda(x) \delta(x) \delta'(x) \equiv A,
		\\
		& \lambda'(x) \delta^2(x)  \equiv B,
	\end{aligned}
\end{equation}
with $A$ and $B$ being two constants.
We study the ODE:
\begin{equation}
	\label{main ODE}
	\begin{aligned}
		&f'' + A f f' \eta - B f^2 = 0.
		\\
		& f(0)=a_0 ,
		\\
		& f'(0)=a_1.
	\end{aligned}
\end{equation}

To obtain a solution which changes its sign, we follow the most straightforward idea by assuming $f(0)<0$ and $f'(0)>0$.

\begin{remark}
	The special solution $u(x,y)=y$ belongs to the family of self-similar solutions that we study. 
	Let $\delta(x) = \lambda(x) = \alpha x^{\frac{1}{3}}$,
	then the solution $u(x,y)= y$ can be written as
	\[
		u(x,y)= \lambda(x) \eta.
	\]
	The function $f(\eta)=\eta$ is indeed a solution to
	\[
		f'' + \frac{\alpha^3}{3} f f' \eta - \frac{\alpha^3}{3} f^2 = 0 . 
	\]
	Another special case is $\delta(x) = \lambda(x) = \beta$. 
	The solution can still be expressed using 
	\[
	u(x,y)= \lambda(x) \eta ,
	\]
	where $f(\eta)=\eta$ is a solution to $f''=0 $.

\end{remark}

Letting $g= f'$, we can transform 
\eqref{main ODE}
to a first-order system.
\begin{equation}
	\label{first order nonlinear system}
	\begin{aligned}
		&f'(\eta)=g(\eta)
		\\
		& g'(\eta) = -A f(\eta)g(\eta) \eta + B f^2(\eta)
		\\
		& (f(0),g(0))= (f_0, g_0) 
	\end{aligned}
\end{equation}

Due to the complexity of this system, we study the behavior of terms $A f(\eta)g(\eta) \eta$ and $B f^2(\eta)$ separately.

If we assume $A=0$ and $B=1$, then 
\eqref{main ODE}
becomes $f''=f^2$. 
Multiplying the equation by $f'$, we obtain
\[
	\prts{f'}^2= \frac{2}{3}f^3 + C_1.
\]
If we choose $C_1=0$ and assume  $f>0$ and $f'>0$, then we obtain 
the solution
\[
	f(\eta)= \frac{6}{\prts{  \eta - C_2  }^2 }.
\]
The solution is the left branch due to the assumption $f'>0$. This implies a blow-up at finite value of $\eta$. Our goal is to obtain a self-similar solution $u$ at least exist in the first quadrant, which requires the solution $f(\eta)$ to exist for $\eta\in [0,+\infty)$. This motivates us to consider the case $A\neq 0$ and $B=0$, or a simpler case $A=1$ and $B=0$.

\section{Behavior of the Term $fg\eta$ When $A>0$ and $B=0$}
\label{section: solve ODE A=1 B=0}

In this section, we prove the existence and uniqueness of a solution to \eqref{first order nonlinear system} when $A>0$ and $B=0$.

For convenience and clarity, we use the equations 
\begin{equation} \label{ODE system in x y}
	\begin{aligned}
		&x'=y,
		\\
		& y' = -A xy t + B x^2,
		\\
		& (x(0),y(0))= (x_0, y_0) ,
	\end{aligned}
\end{equation}
when there is no confusion with \eqref{main PDE}

Assume $A>0$ and $B=0$. 
Without loss of generality, we may assume $A=1$, as a change of variables $\overline{x}=Ax$, $\overline{y}=Ay$ absorbs the constant $A$.  
Thus, we consider the problem
\begin{equation} \label{ODE system in x y, B=0}
	\begin{aligned}
		&x'=y,
		\\
		& y' = - xy t ,
		\\
		& (x(0),y(0))= (x_0, y_0) ,
	\end{aligned}
\end{equation}
with initial values  $x_0<0$ and $y_0>0$.
The main idea is to extend the solution step by step using the Picard--Lindel\"of theorem. See \cite{Teschl.ODE} Theorem 2.2 for details.

\subsection{Solution in the second quadrant}

Consider the closed cylinder 
\[
	(x,y,t)\in U_0 \times [-1,1]:= B((x_0,y_0);\delta) \times [-1,1],
\] 
with $\delta$ sufficiently  small such that
the closed disc %$\overline{B((x_0,y_0);\delta_0)}$
$ \overline{U_0} $ remains in the second quadrant.
The functions
$F(x,y,t)= y$  and  $G(x,y,t) = -xyt$ 
are Lipschitz continuous in $U_0\times [-1,1]$ with respect to $(x,y)$ and uniformly in $t$,  which implies the local existence of a unique $C^1$ solution 
inside the region 
\[
	%B((x_0,y_0);\delta_0) 
	U_0 \times [-\varepsilon_0 ,\varepsilon_0]
\]
for some sufficiently small $\varepsilon_0$.

Let $t_1=\varepsilon_0$ be the new initial time with initial value $(x_1,y_1):=(x(t_1), y(t_1))$. Notice that $x_1<0$ and $y_1>0$.
Consider the cylinder
\[
	U_1 \times \left[ t_1 -1,t_1 +1 \right] := B((x_1,y_1);\delta_1) \times \left[ t_1 -1,t_1 +1 \right]
\]
with $\delta_1$ chosen such that the closed disc %$\overline{B((x_1,y_1);\delta_1)}$
 $\overline{U_1}$
 remains in the second quadrant. 
Similarly, we obtain the existence and uniqueness of a solution in 
\[
 U_1 \times \left[ t_1 -\varepsilon_1 ,t_1 + \varepsilon_1 \right]
\]
for sufficiently small $\varepsilon_1$.

Let $t_2:= t_1 + \varepsilon_1$, and by repeating the steps above, we obtain a strictly increasing sequence $t_k$ such that
the unique solution exists on $[0,t_k]$ 
and 
$U_k$ remains in the second quadrant for all $k\in \mathbb{N}$.

\begin{figure}[h]
	\centering
	\begin{tikzpicture}[scale=1]
		\draw[->] (-4,0) -- (0.8,0) node[right] {$x$};
		\draw[->] (0,-0.4) -- (0,4) node[right] {$y$};
		
		\draw[thick, blue] (-3.2,1) .. controls (-2.8,1.6) and (-2.4,2) .. (-1.2,2.4);
		
		\filldraw[black] (-3.2,1) circle (0.04);
		\filldraw[black] (-2.76,1.56) circle (0.04);
		\filldraw[black] (-2.2,1.98) circle (0.04);
		\filldraw[black] (-1.2,2.4) circle (0.04);
		
		\draw[black] (-1.2,2.4) circle (0.8);
		\node at (-1.2,2.72) {$U_3$};

	\end{tikzpicture}
	\caption{The neighborhoods $U_k$ remain in the second quadrant. Both $x(t)$ and $y(t)$ are increasing.}
	
	\label{figure: extend solution in second quadrant}
\end{figure}
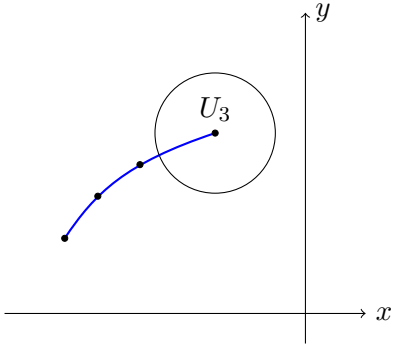

Let $T_1:= \sup_k t_k$. Then for all $t\in[0,T_1)$, we have $x(t)<0$ and $y(t)>0$.
This implies that $x'>0$ and $y'\geq 0$ by \eqref{ODE system in x y, B=0}. In fact, by our construction, we have $y'>0$ on $[t_1, T_1)$. We claim that the only possible behavior of the trajectory is to intersect the  $y$-axis at time $T_1$.

\begin{lemma} \label{lemma: solution in second quadrant}
	The trajectory intersects the positive $y$-axis in finite time, 
	i.e.,
	\[
		T_1<\infty, \quad \lim_{t\to T_1 } x(t)=0, \quad \text{ and }
		\quad 
		0< \lim_{t\to T_1 } y(t) <+\infty.
	\] 
\end{lemma}

\begin{proof}
	Since $y'(t)\geq 0$, by \eqref{ODE system in x y, B=0} we have $x'(t) = y(t)\geq y_0 > 0$ on $[0,T_1)$. 
	By the mean value theorem, we have for all $t_k$ that
	\[
		\abs{x_0}\geq x(t_k) - x_0 = x'(\xi_k) t_k \geq y_0 t_k,
	\]
	which implies 
	\[
		0< T_1 = \sup_k t_k \leq \frac{\abs{x_0}}{y_0}  <\infty.
	\]
	
	Notice that $x(t)$ and $y(t)$ are both monotone increasing on $[0, T_1)$ and $x(t)< 0$. 
	We claim that $y(t)$ is bounded on $[0, T_1)$.
	%Assume $y(t)\to +\infty$ as $t\to T_1^-$. Then the slope $y'(t)/ x'(t) \to +\infty$.
 	From \eqref{ODE system in x y, B=0}, we have for all $t\in[0,T_1)$ that 
 	\[
 		\frac{y'(t)}{x'(t)} = \frac{-x(t)y(t)t}{y(t)} = -x(t)t 
 		\leq 
 		\abs{x_0} T_1
 		\leq
 		\frac{x_0^2}{y_0}.
 	\]
	Using the Cauchy mean value theorem, we have
	\[
		\frac{y(t)-y_0}{x(t)-x_0} = \frac{y'(\xi)}{x'(\xi)} \leq \frac{x_0^2}{y_0}
	\]
	for some $\xi\in (0,t) $, which implies
	\[
		y(t) \leq y_0 + \frac{x_0^2}{y_0} ( x(t)-x_0 )
		\leq  y_0 + \frac{\abs{x_0}^3}{y_0} .
	\]

	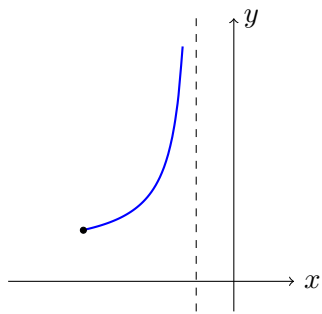
\begin{figure} [h]
		\centering
		\begin{tikzpicture}[scale=1]
			
			\draw[->] (-3,0) -- (0.8,0) node[right] {$x$};
			\draw[->] (0,-0.4) -- (0,3.5) node[right] {$y$};
			
			\draw[dashed] (-0.5, -0.4) -- (-0.5, 3.5);
			
			\draw[domain=-2:-0.68, smooth, variable=\x, blue, thick] plot ({\x}, {-0.5/(\x + 0.5) +0.35});
			
			\filldraw[black] (-2,0.68) circle (0.04);

		\end{tikzpicture}
		
		\caption{The only way for the solution to be unbounded is $x\to a $ and $y\to +\infty$, which leads to a contradiction in the behavior of tangent lines.}
		
		\label{figure: assume blow up in second quadrant}
		
	\end{figure}

	Now we claim that
	\[
		x(T_1^-):= \lim_{t\to T_1^-} x(t) = 0.
	\]
	Assume $x(T_1^-) < 0$. Then  $(x(T_1^-), y(T_1^-) )$ is in the second quadrant. 
	We can use it as an initial value one more time to extend $T_1$, which contradicts with the definition of $T_1$.

\end{proof}

\subsection{Solution in the first quadrant}
Now we study the solution when $t\geq T_1$.
We still use notations $U_k$ as long as there is no risk of ambiguity.
Let 
\[
 U_0 \times [T_1-1, T_1 + 1]
 := B((x(T_1),y(T_1));\delta_0) \times [T_1-1, T_1 + 1],
\] 
with $\delta_0$ chosen such that 
\[
	 \overline{U_0} \subseteq \{ y>0\}.
\]
Then we can obtain a solution inside the cylinder
\[
	U_0 \times [T_1-\varepsilon_0, T_1 + \varepsilon_0]
\]
for some $\varepsilon_0$.
Moreover, $x(t)$ is strictly increasing by \eqref{ODE system in x y, B=0}.
This step pushes the solution into the first quadrant, as shown in Figure \ref{figure: push solution to first quadrant}.

\begin{figure}[h]
	\centering
	\begin{tikzpicture}[scale=1]
		
		% Axes
		\draw[->] (-2,0) -- (1.5,0) node[right] {$x$};
		\draw[->] (0,-0.4) -- (0,3.5) node[right] {$y$};
		
		% Plot y = cos(x + 0.5) + 0.5 over x in [-1, 0]
		\draw[domain=-1.5:0.4, smooth, variable=\x, thick, blue] 
		plot ({\x}, {1.5*cos(deg(\x )) + 0.5});

		\filldraw[black] (-1.5,0.6) circle (0.04);
		
		\filldraw[black] (0,2) circle (0.04);
		
		\filldraw[black] (0.4,1.88) circle (0.04);

		\node at (-0.7,2.3) {$t=T_1$};

		\node at (0.7,1.5) {$t=t_1$};
		
		\draw[black] (0,2) circle (0.8);

	\end{tikzpicture}
	
	\caption{Pushing the solution into the first quadrant.}
	
	\label{figure: push solution to first quadrant}
	
\end{figure}
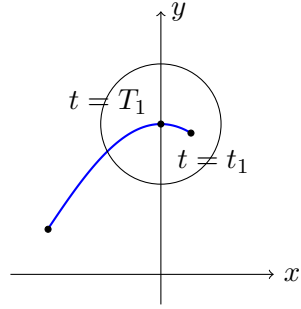

Now we study the behavior of the solution when it's completely in the first quadrant $\{x>0, y>0\}$.
Let $t_1 = T_1 + \varepsilon_0 $ and let $(x_1, y_1):= (x(t_1), y(t_1)) $.
We now choose
\[
U_1\times [t_1-1, t_1 + 1] := B((x_1,y_1);\delta_1) \times [t_1-1, t_1 + 1],
\] 
such that the closed disc $\overline{U_1}$ is in the first quadrant $\{ x>0, y>0 \}$.
We can find the unique solution in the region
\[
	U_1 \times [t_1-\varepsilon_1, t_1 + \varepsilon_1 ]
\]
for some $\varepsilon_1$.
Let $t_2 := t_1 + \varepsilon_1$,
and repeating the steps above, we can obtain a strictly increasing sequence $t_k$ 
such that 
the unique solution exists on $[T_1,t_k]$ and 
$ \overline{U_k} \subseteq \{ x>0, y>0\}$
for all $k\in\mathbb{N}$.

Let $T_2:= \sup_k t_k$. It is supposed to be the time that the solution intersects the positive  $x$-axis.

\begin{lemma}
	\label{lemma: solution in first quadrant}
	The solution exists on $[T_1,\infty)$, i.e.
	$T_2 = +\infty$. 
	The trajectory converges to a finite point $(a,0)$ on $x$-axis with $a>0$. 
\end{lemma} 

\begin{proof}
	Assume $T_2 < +\infty$. Since $x(t)$ is increasing on $[T_1,T_2)$, we can divide the cases by whether $x(t)$ is bounded.
	
	Case 1: Suppose $x(t)\to +\infty$ as $t\to T_2^-$, as shown in Figure \ref{figure: T2 finite x unbounded}.
	From \eqref{ODE system in x y, B=0}	we have
	\[
	\lim_{t\to T_2^-}\frac{y'(t)}{x'(t)} 
	= \lim_{t\to T_2^-} -x(t) t 
	= -\infty \cdot T_2 = -\infty.
	\] 
	Thus, we can find $\delta>0$ such that
	\[
		\frac{y'(t)}{x'(t)} \leq -1
	\]
	in $[T_2 - \delta, T_2)$.
	For all $t\in (T_2-\delta, T_2)$,
	by the Cauchy mean value theorem, we have
	\[
		\frac{y(t)-y(T_2- \delta)}{x(t)-x(T_2- \delta)}
		=
		\frac{y'(\xi)}{x'(\xi)}
		\leq -1
	\]
	for some $\xi\in (T_2-\delta , t)$.
	Since $y(t)-y(T_2- \delta) \leq 0$, we have
	\[
		0\leq y(t)
		\leq y(T_2- \delta) + x(T_2- \delta) - x(t) ,
	\]
	which is a contradiction as $x(t)\to +\infty$.

	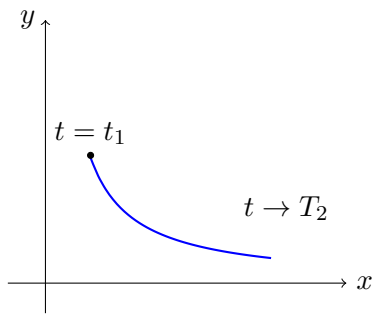
\begin{figure}[h]
		\centering
		\begin{tikzpicture}[scale=1]
			
			% Axes
			\draw[->] (-0.5,0) -- (4,0) node[right] {$x$};
			\draw[->] (0,-0.4) -- (0,3.5) node[left] {$y$};

			\draw[domain=0.6:3, smooth, variable=\x, thick, blue] 
			plot ({\x}, { 1/ \x  });

			\filldraw[black] (0.6,1.7) circle (0.04);
			
			\node at (0.6,2) {$t= t_1$};

			\node at (3.2,1) {$t\to T_2$};
			
			%\draw[->, dashed] (3.2,0.3) -- (3.6, 0.3);

		\end{tikzpicture}
		
		\caption{If $T_2<\infty$ and $x(t)$ is unbounded, a contradiction on the slope will occur. }
		
		\label{figure: T2 finite x unbounded}
		
	\end{figure}

	Case 2: Suppose $x(t)$ is bounded on $[T_1, T_2)$, as shown in Figure \ref{figure: T2 finite x bounded}.
	In this case, both $x(t)$ and $y(t)$ are monotone and bounded.
	Since $x(t_1)>0$, we have
	\[
		\begin{aligned}
			&
			x(T_2^-) := \lim_{t\to T_2^-} x(t) >0 ,
			\\
			&
			y(T_2^-) := \lim_{t\to T_2^-} y(t) \geq 0.
		\end{aligned}
	\]
	
		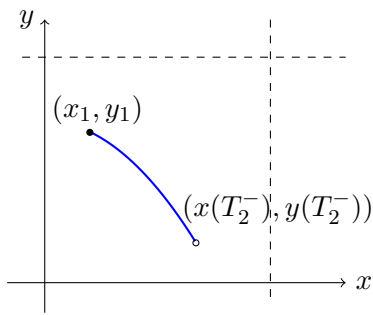
\begin{figure}[h]
		\centering
		\begin{tikzpicture}[scale=1]
			
			% Axes
			\draw[->] (-0.5,0) -- (4,0) node[right] {$x$};
			\draw[->] (0,-0.4) -- (0,3.5) node[left] {$y$};

			\draw[ dashed] (-0.3,3) -- (4, 3);
			
			\draw[ dashed] (3,3.5) -- (3, -0.2);

			\draw[domain=0.6:2, smooth, variable=\x, thick, blue] 
			plot ({\x}, { - 0.4* (\x)^2 + 2.15  });

			\filldraw[black] (0.6,2) circle (0.04);
			
			\node at (0.7,2.3) {$(x_1, y_1)$};
			
			\draw[black] (2.01,0.53) circle (0.04);

			% \node at (2.7,1) {$t\to T_2$};
			
			\node at (3.1,1) {$(x(T_2^-), y(T_2^-))$};

		\end{tikzpicture}
		
		\caption{If $T_2<\infty$ and $x(t)$ is bounded, the trajectory has a limit point. No matter $ y(T_2^-)>0$ or $ y(T_2^-)=0$, a contradiction will arise. }
		
		\label{figure: T2 finite x bounded}
		
	\end{figure}
	
	If $y(T_2^-)>0$, then we can use $( x(T_2^-), y(T_2^-) )$ as the initial value to extend $[T_1,T_2)$ using the same idea as in Lemma \ref{lemma: solution in second quadrant}. This contradicts the definition of $T_2$.
	
	If $y(T_2^-)=0$, then it contradicts with the uniqueness of solution. Notice that $x(t)\equiv x(T_2^-), y(t)\equiv 0$  also provides a solution.
	Consider the cylinder 
	\[
		 B( (x(T_2^-), 0);\delta )\times [T_2-1, T_2 ].
	\]
	There should be a unique solution in
	\[
		B( (x(T_2^-), 0);\delta )\times [T_2-\varepsilon, T_2 ],
	\]
	 for some $\varepsilon>0$, 
	 as $F(x,y,t)= y$  and  $G(x,y,t) = -xyt$ are Lipschitz continuous with respect to $(x,y)$ and uniformly in $t$.

	 Therefore, we conclude that $T_2 =+\infty$. Since $x(t)$ is monotone increasing, and $y(t)$ is monotone decreasing and bounded, it remains to prove
	 \[
	 \begin{aligned}
	 	&
	 	 a:= \lim_{t\to +\infty } x(t)   <\infty ,
	 	\\
	 	&
	 	 b:= \lim_{t\to +\infty} y(t) = 0.
	 \end{aligned}
	 \]
	Assume $a= +\infty$.
From \eqref{ODE system in x y, B=0}
we have
\[
\lim_{t\to +\infty} \frac{y'(t)}{x'(t)} 
=  \lim_{t\to +\infty} -x(t) t = -\infty.
\] 
Then there exists $M>0$ such that 
\[
	\frac{y'(t)}{x'(t)} \leq -1
\]
for all $t\in [M,+\infty)$.
Using the Cauchy mean value theorem, we have
\[
	\frac{y(t)- y(M)}{x(t)-x(M)} = \frac{y'(\xi)}{x'(\xi)} \leq -1
\]
for all $t\in(M,+\infty)$,
which implies 
\[
	 y(t)\leq y(M)+ x(M) - x(t) \to -\infty
\]
as $t\to+\infty$.
This contradicts the fact that $y\geq 0$.
Thus, we have $a<+\infty$.

Now assume that $b>0$. From $y'= -xyt$, we have
\[
	\lim_{t\to+\infty } y'(t) = -\infty,
\]
which is a contradiction as $y(t)$ is decreasing and $0\leq y(t)\leq y_1$. Thus, we have $b=0$.
This completes the proof.
\end{proof}

Using Lemma \ref{lemma: solution in second quadrant} and Lemma \ref{lemma: solution in first quadrant}, we can obtain a unique solution that crosses the $y$-axis and converges to the $x$-axis, as shown in Figure \ref{figure: whole solution in two quadrants}.

\begin{figure}[h]
	\centering
	\begin{tikzpicture}[scale=1]
		
		% Axes
		\draw[->] (-2,0) -- (2.5,0) node[right] {$x$};
		\draw[->] (0,-0.4) -- (0,3.5) node[right] {$y$};
		
		% Points
		\filldraw[black] (-1.5,0.6) circle (0.04);
		\filldraw[black] (0,2) circle (0.04);
		
		\node at (-0.7,2.3) {\footnotesize $t=T_1$};
		
		% Left wavy or cosine part
		\draw[domain=-1.5:0, smooth, variable=\x, thick, blue] 
		plot ({\x}, {1.5*cos(deg(\x )) + 0.5});
		
		% Replace this part with perfect arc
		\draw[thick, blue] (0,2) arc[start angle=90, end angle=0, radius=2];
		
		\draw[black] (2,0) circle (0.04);
		
		% Asymptote 
		\draw[dashed] (2,-0.4) -- (2,3.5);
		
		\draw[dashed] (-2,2) -- (3,2);
		
		\node at (1.1,0.4) {\footnotesize  $t\to +\infty$};
		
	\end{tikzpicture}
	
	\caption{A unique solution exists given generic initial values $x_0<0$ and $y_0>0$. }
	
	\label{figure: whole solution in two quadrants}

\end{figure}
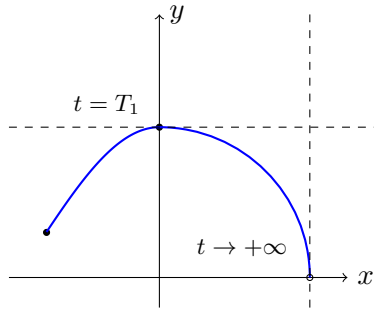

Therefore, we have obtained the following result for the original second-order ODE. The structure of the equation increases the smoothness  of the solution $f(\eta)$ from $C^1$ to $C^\infty$. 
\begin{proposition} \label{proposition: ODE of f, existence}
	Given any $a_0 <0$ and $a_1>0$, the equation 
	\begin{equation}
		\label{ODE of f term, existence result}
		\begin{aligned}
			&
			f''(\eta) = - f(\eta) f'(\eta) \eta,
			\\
			&
			f(0)=a_0 ,
			\quad 
			f'(0)= a_1
		\end{aligned}
	\end{equation}
	has a unique  solution $f\in C^\infty[0,+\infty)$, which is increasing and 
	\[
		0< \lim_{\eta\to+\infty}f(\eta) <+\infty.
	\]
\end{proposition}

\section{The Self-Similar Solution}
\label{Section: self-similar solution}

In this section, we use the solution $f(\eta)$ in \eqref{ODE of f term, existence result} to recover the information about $u(x,y)$.

Since we assume $A=1$ and $B=0$, we solve the equations in \eqref{simplify to ODE, constant terms}
\[
		\begin{aligned}
		& \lambda(x) \delta(x) \delta'(x) = A = 1,
		\\
		& \lambda'(x) \delta^2(x)  = B = 0.
	\end{aligned}
\]
From the definition of the similarity variable $\eta = y/\delta(x)$, it is natural to assume $\delta \neq 0$, which implies $\lambda'=0$, i.e.,
\[
	\lambda(x)=c_1.
\]
From the first equation, we have
\[
	\delta(x) \delta'(x) = \frac{1}{c_1},
\]
which implies
\[
	\delta^2(x)  = \frac{2}{c_1} x + c_2.
\]
For simplicity, we choose $c_1=2$ and $c_2=0$, which implies 
\[
	\lambda(x)=2, \quad  \delta(x)=\sqrt{x}.
\]
Therefore, we have obtained a self-similar solution 
\begin{equation}
	u(x,y)= 2 f\prts{ \frac{y}{\sqrt{x}}}
\end{equation}
defined on $\{ x>0, y\geq 0 \}$.
The solution can be expressed by contour lines as in Figure \ref{contour lines of self-similar solution}.

\begin{figure}[h]
	
	\centering
	\begin{tikzpicture}
		\draw[->] (0,0) -- (4,0) node[right] {$x$};
		\draw[->] (0,0) -- (0,4) node[above] {$y$};
		\draw[domain=0:3.7] plot (\x,{sqrt(\x)});
		\node at (3,1.3) {$u = f(1)$};
		\draw[domain=0:3.5] plot (\x,{2*sqrt(\x)});
		\node at (3,2.9) {$u = f(2)$};
		
		\draw[domain=0:2] plot (\x,{3*sqrt(\x)});
		\node at (2.4,4) {$u = f(3)$};

	\end{tikzpicture}
	\caption{Contour lines of the self-similar solution}
	\label{contour lines of self-similar solution}
	
\end{figure}
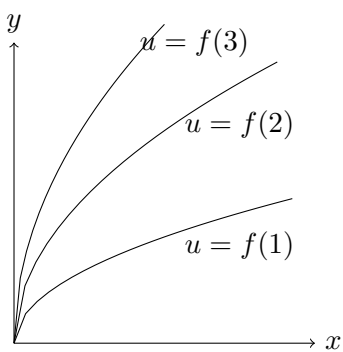

When $(x,y)\to (0,y_0)$ with $y_0>0$, we have $\eta \to +\infty$ and
\[
\lim_{(x,y)\to(0,y_0)} u(x,y)= 2 \lim_{\eta\to +\infty }  f\prts{ \eta},
\]
which enables us to extend the solution continuously to the positive  $y$-axis.

Since $f(\eta)$ is monotone increasing, we can verify by direct computation that $u_x \geq 0$ and $u_y \geq 0$.
In this case, it is much easier to obtain some estimates near the interface.

Consider the region $[X_1,X_2]\times[Y_1,Y_2]$ such that the boundary $\{u=0\}$ intersects  the left and right sides, as shown in Figure \ref{figure: self similar solution local region}. Then we obtain a problem similar to \cite{Dalibard.version2025}, which is strongly related to the study on Prandtl boundary layer equation in \cite{Masmoudi.Reversal, Masmoudi.Higher-regularity}.

\begin{figure}[h]
	
	\centering
	
	\begin{tikzpicture}
		\draw[->] (0,0) -- (4,0) node[right] {$x$};
		\draw[->] (0,0) -- (0,3) node[above] {$y$};
		
		\draw[domain=0:3.7] plot (\x,{sqrt(\x)});
		\node at (3.8,1.6) {$u=0$};
		
		\draw[blue, thick] (0.7,0.5) rectangle (2.3,1.8);
		
		\draw[dashed] (0.7,0.5) -- (0.7,0) node[below] {$X_1$};
		\draw[dashed] (2.3,1.8) -- (2.3,0) node[below] {$X_2$};
		\draw[dashed] (0.7,0.5) -- (0,0.5) node[left] {$Y_1$};
		\draw[dashed] (2.3,1.8) -- (0,1.8) node[left] {$Y_2$};
		
		\node[above left] at (0.8, 0.6) {$A$};
		\node[above right] at (2.2, 1.1) {$B$};
	\end{tikzpicture}

	\caption{Solution in a local region.}
	\label{figure: self similar solution local region}
	
\end{figure}
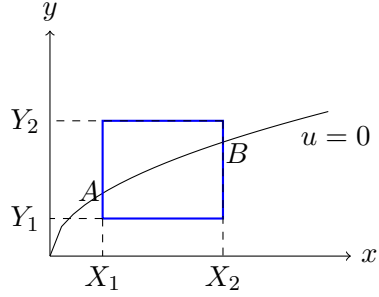

For this special self-similar solution, we can expect to get better estimates. For simplicity, we use $A$ and $B$ to denote their $x$ or $y$ coordinates as long as there is no confusion. Multiply the equation by $u$, and integrate over $[X_1,X_2]\times[Y_1,Y_2]$. Then we obtain
\[
	\begin{aligned}
		&
		0 = \int_{X_1}^{X_2}\int_{Y_1}^{Y_2} \prts{uu_x - u_{yy}} u dydx 
		\\
		&= \int_{X_1}^{X_2}\int_{Y_1}^{Y_2} u^2u_x dydx
		+ \int_{X_1}^{X_2}\int_{Y_1}^{Y_2} - u_{yy}u dydx
		\\
		& = I_1 + I_2.
	\end{aligned}
\]
For $I_1$ we have
\[
	\begin{aligned}
		&
		I_1 = 
		\int_{Y_1}^{Y_2} \int_{X_1}^{X_2} \prts{u^3}_x dx dy
		=
		\int_{Y_1}^{Y_2} u^3(X_2,y) dy
		-
		\int_{Y_1}^{Y_2} u^3(X_1,y)  dy
		\\
		&
		=
		\int_{Y_1}^{B} u^3(X_2,y)  dy
		+\int_{B}^{Y_2} u^3(X_2,y)  dy
		-
		\int_{Y_1}^{A} u^3(X_1,y)  dy
		- \int_{A}^{Y_2} u^3(X_1,y) dy.
	\end{aligned}
\] 
For $I_2$ we have
\[
	\begin{aligned}
		&
		I_2 
		=
		\int_{X_1}^{X_2}\int_{Y_1}^{Y_2}\prts{ - \prts{u_{y}u}_y + u_y^2 }dydx
		\\
		&
		=
		-\int_{X_1}^{X_2} u_{y}(x, Y_2)u (x, Y_2)   dx
		+ \int_{X_1}^{X_2}   u_{y}(x, Y_1) u (x, Y_1)  dx
		+
		\int_{X_1}^{X_2}\int_{Y_1}^{Y_2}   u_y^2 dydx.
	\end{aligned}
\]
Thus, we have
\[
\begin{aligned}		
	&
	\int_{X_1}^{X_2}\int_{Y_1}^{Y_2}   u_y^2 dydx
	+\int_{B}^{Y_2} u^3(X_2 ,y)  dy
	+
	\int_{Y_1}^{A} -u^3(X_1 ,y)  dy
	\\
	&=
	 \int_{Y_1}^{B} -u^3(X_2 ,y)  dy
	+ \int_{A}^{Y_2} u^3(X_1 ,y)  dy
	\\
	& \quad + \int_{X_1}^{X_2} u_{y}(x, Y_2)u (x, Y_2)   dx
	+ \int_{X_1}^{X_2}  - u_{y}(x, Y_1) u (x, Y_1)  dx,
\end{aligned}
\]
which is equivalent to
\[
\begin{aligned}		
	&
	\int_{X_1}^{X_2}\int_{Y_1}^{Y_2}   u_y^2 dydx
	+\int_{B}^{Y_2} \abs{u}^3(X_2 ,y)  dy
	+
	\int_{Y_1}^{A} \abs{u}^3(X_1 ,y)  dy
	\\
	&=
	\int_{Y_1}^{B} \abs{u}^3(X_2 ,y) dy
	+ \int_{A}^{Y_2} \abs{u}^3(X_1 ,y)  dy
	\\
	& \quad + \int_{X_1}^{X_2} u_{y}(x, Y_2)\abs{u} (x, Y_2)   dx
	+ \int_{X_1}^{X_2}   u_{y}(x, Y_1) \abs{u} (x, Y_1)  dx.
\end{aligned}
\]

\section{Acknowledgment}
The author would like to thank Professor Sijue Wu for valuable suggestions on the study of self-similar solutions.

% 2025-7-7 morning. checked the whole paper for the first time.
% 2025-7-18 afternoon. checked the whole paper again.

%------------theory ends here----------------------------------------------------------------

%------------below is reference-------------------------------------------------------

%\printbibliography
\bibliography{ref_bl}{}
\bibliographystyle{abbrv} %abbrv.bst %plain

\end{document}